\documentclass[reqno,12pt,letterpaper]{amsart}

\usepackage{amsmath,amssymb,amsthm,graphicx,mathrsfs,url}
\usepackage[usenames,dvipsnames]{color}
\usepackage[colorlinks=true,linkcolor=Red,citecolor=Green]{hyperref}
\usepackage{amsxtra,esint}
\usepackage{mathtools}
\usepackage{hyperref}
\usepackage{MnSymbol}

\usepackage{dsfont}
\usepackage{commath}
\usepackage{comment}
\usepackage{cleveref}
\usepackage{enumerate}

\usepackage{float}
\usepackage[font=small,labelfont=bf]{caption}

\def\?[#1]{\textbf{[#1]}\marginpar{\Large{\textbf{??}}}}

\newtheorem{prop}{Proposition}[section]
\newtheorem{defi}[prop]{Definition}
\newtheorem*{theorem*}{Theorem}
\newtheorem{theorem}[prop]{Theorem}

\newtheorem{lemma}[prop]{Lemma}
\newtheorem{corr}[prop]{Corollary}
\newtheorem{rem}[prop]{Remark}

\numberwithin{equation}{section}

\newtheorem{question}{Question}
\newtheorem*{main-question}{Main Question}

\newcommand{\N}{\mathbb{N}}

\newcommand{\R}{\mathbb{R}}

\newcommand{\C}{\mathbb{C}}
\renewcommand{\phi}{\varphi}

\newcommand{\uno}{\mathds{1}}

\newcommand{\eps}{\varepsilon}

\DeclareMathOperator{\spn}{span}

\begin{document}

\title{Isometric embeddings into $C(K)$-spaces doing stable phase retrieval}

\author[Garc\'ia-S\'anchez]{Enrique Garc\'ia-S\'anchez}
\address{Instituto de Ciencias Matem\'aticas (CSIC-UAM-UC3M-UCM)\\
Consejo Superior de Investigaciones Cient\'ificas\\
C/ Nicol\'as Cabrera, 13--15, Campus de Cantoblanco UAM\\
28049 Madrid, Spain.
\newline
	\href{https://orcid.org/0009-0000-0701-3363}{ORCID: \texttt{0009-0000-0701-3363} } }
\email{enrique.garcia@icmat.es}

\author[de Hevia]{David de Hevia}
\address{Instituto de Ciencias Matem\'aticas (CSIC-UAM-UC3M-UCM)\\
Consejo Superior de Investigaciones Cient\'ificas\\
C/ Nicol\'as Cabrera, 13--15, Campus de Cantoblanco UAM\\
28049 Madrid, Spain.
\newline
\href{https://orcid.org/0009-0003-5545-0789}{ORCID: \texttt{0009-0003-5545-0789}}}
\email{david.dehevia@icmat.es // davhevia@ucm.es}



\begin{abstract}
Motivated by a question posed by Freeman, Oikhberg, Pineau and Taylor, we prove that if $K$ is a compact Hausdorff space with $K^{(\alpha)}\neq\varnothing$, where $2<\alpha<\omega$, then $C[1,\omega^\alpha]$ isometrically embeds into $C(K)$ doing stable phase retrieval (SPR). We also show that the latter cannot be extended to the case $\alpha=2$.
\end{abstract}

\maketitle

\section{Introduction}

In many areas of physics and engineering, such as X-ray crystallography, astronomy, quantum mechanics, or speech recognition, it is often necessary to \textit{recover} a signal from the measurement of a \textit{phaseless quantity} associated with it, such as the intensity of an electromagnetic or acoustic wave, or the probability density of the wave function of a quantum system. In these situations, the sign or phase of the signal (depending on whether it has a real or complex nature) can contain relevant information impossible to obtain directly from measurement devices, which can only record the absolute value. Therefore, a mathematical approach is needed for its recovery. These kinds of problems are known as \textit{phase retrieval problems} (see the surveys \cite{surveyGKR,surveyJEH} for an overview of the topic). It should be noted that, given a fixed function, any scalar of modulus one multiplied by the function will always have the same absolute value, so there will always be an unavoidable ambiguity in the process of recovering a function from its absolute value. Moreover, many of the examples originate from experimental disciplines, where empirical errors can perturb the original signal. Therefore, a study of the stability of the recovery process is also necessary (\textit{stable phase retrieval problems}).\\

For modeling phase retrieval problems, \emph{function spaces} offer a natural framework. These include, for instance, \emph{$L_p$-spaces} or \emph{spaces of continuous functions on compact Hausdorff spaces} (for short, \emph{$C(K)$-spaces}). However, it should be noted that the formulation of (stable) phase retrieval problems only requires a linear and (potentially) a normed structure, and an operation of absolute value or modulus. This is the reason why the theory of Banach lattices, which are an abstract generalization of function spaces whose properties are defined precisely by these three factors (see \cite{AliprantisBurkinshaw, LindenstraussTzafririVol2, MeyerNieberg} for standard notation and basic definitions), has provided very useful tools to tackle these questions \cite{CGH,FOPT}. This approach has been successfully applied in $L_p$-spaces to construct infinite-dimensional subspaces doing stable phase retrieval \cite{CDFF,CPT,FOPT}. In contrast, and despite of being an extensively studied family of Banach lattices, $C(K)$-spaces have received \emph{little attention} in the context of stable phase retrieval, aside from what is found in \cite{FOPT}. Thus, the objective of this work is to continue the study of stable phase retrieval in $C(K)$-spaces initiated in \cite{FOPT}.\\

\subsection{Motivation and preliminary results} As explained in \cite{FOPT}, Banach lattices provide a suitable general setting for many phase retrieval problems, such as Fourier phase retrieval, Gabor phase retrieval or wavelet phase retrieval, and even other inverse problems such as audio declipping \cite{AFRT}.
In these examples, the problem consists in uniquely recovering an element $f$ of a Banach space $F$ from $|Tf|$, which is the modulus of the image of $f$ by an invertible operator $T$ from $F$ into a suitable target Banach lattice $X$. In the case of Fourier phase retrieval, $F=X=L_2$ and $T$ is the Fourier transform, whereas for the Gabor and wavelet phase retrieval, $T$ corresponds to the analysis operator of a certain frame on the Banach space $F$. These problems can be split into two parts: the first one is to retrieve $Tf$ (up to a global phase) from $|Tf|$ within the subspace $E=T(F)\subseteq X$, and the second one is to invert $T$, a problem which is much better understood. This motivates the following definition:

\begin{defi}[{\cite[Definition 3.1]{FOPT}}]\label{defi: PR in BL}
    A subspace $E$ of a Banach lattice $X$ does \emph{phase retrieval (PR)} if whenever $x,y\in E$ satisfy $|x|=|y|$, it follows that $x=\lambda y$ for some unimodular scalar $\lambda$.
\end{defi}

Consider the equivalence relation $\sim$ given by $x\sim y$ if and only if $x=\lambda y$ for some unimodular scalar $\lambda$. The fact that a subspace $E\subseteq X$ does PR is equivalent to the injectivity, or equivalently, the invertibility, of the map that sends any equivalence class of $E/\sim$ to the absolute value in $X$ of a representative of the class. However, in applications it is also necessary to understand \emph{how stable} is the inverse map under small perturbations, such as experimental errors. That is, studying if the inverse map has some kind of continuity, and more specifically, Lipschitz continuity:

\begin{defi}[{\cite[Definition 3.2]{FOPT}}]\label{defi: SPR in BL}
    A subspace $E$ of a  Banach lattice $X$ does \emph{$C$-stable phase retrieval} (\emph{$C$-SPR}, for short) if
    \begin{equation}\label{eq: SPR}
    \min_{|\lambda|=1}\|x-\lambda y\|\leq C\, \bigl\||x|-|y|\bigr\|
\end{equation}
    for every $x,y\in E$. We will say that $E$ does \emph{stable phase retrieval (SPR)} if it does $C$-SPR for some constant $C\geq 1$.
\end{defi}

The last two definitions can be considered for both real and complex scalars; however, there are some noteworthy differences between the two cases. Specifically, recall that in the real setting \emph{the main and only obstruction} that can cause a subspace $E$ of a Banach lattice $X$ to fail PR is the existence of pairs of disjoint vectors in $E$, that is, non-zero functions $f,g\in E$ such that $|f|\land |g|=0$ (or, equivalently, $fg=0$ if $X$ is a function space). This fact appears implicitly in \cite{FOPT} and its proof can be found in \cite[Proposition 2.1]{CGH}. Motivated by the aforementioned characterization of phase retrieval, the authors in \cite{FOPT} consider the notion of \emph{almost disjoint pairs} with the aim of studying SPR:

\begin{defi}\label{defi: almost disjoint pairs}
Let $E$ be a subspace of a Banach lattice $X$. Given $\eps>0$, we say that $f,g\in S_{E}$ is a \emph{(normalized) $\eps$-almost disjoint pair} if $\||f|\wedge|g|\|\leq \eps$. If $E$ contains $\varepsilon$-almost disjoint pairs for every $\varepsilon>0$, we say that $E$ contains \emph{almost disjoint pairs}.
\end{defi}

In fact, almost disjoint pairs were used in \cite[Theorem 3.4]{FOPT} to provide a powerful characterization of SPR for real scalars, which reads as follows: the subspaces of a Banach lattice which do stable phase retrieval are precisely those that do not contain almost disjoint pairs. We present below a refinement of this result due to E. Bilokopytov \cite[Proposition 3.4]{Eugene}. 

\begin{theorem}\label{thm: real stable phase retrieval}
    Let $E$ be a subspace of a real Banach lattice $X$, and let $C>0$ be a constant. $E$ does $C$-SPR if and only if $E$ does not contain any $\frac{1}{C}$-almost disjoint pair. In particular, $E$ does SPR if and only if it does not contain almost disjoint pairs.
\end{theorem}

It is also worth mentioning that a complex version of the above statement was established in \cite[Theorem 2.7]{CGH}. In the complex case, we can only guarantee one of the implications of the previous theorem: if a subspace of a complex Banach lattice does SPR, then it does not contain almost disjoint pairs. More precisely, one can show the following \cite[Proposition 2.5]{CGH}:

\begin{prop}\label{prop: failing complex SPR}
Let $E$ be a subspace of a complex Banach lattice $X$ and $\varepsilon>0$. If there exist $u,v\in S_E$ such that $\||u|\land |v|\|<\varepsilon$, then $E$ fails $\frac{1}{\sqrt{2}\varepsilon}$-SPR.    
\end{prop}

This characterization of stable phase retrieval in terms of the non-existence of almost disjoint pairs is one of the key ingredients used by Freeman, Oikhberg, Pineau and Taylor to show that a $C(K)$-space contains an infinite-dimensional subspace doing SPR if and only if $K'$ is infinite \cite[Theorem 6.1]{FOPT}, where $K'$ stands for the set of non-isolated (or accumulation) points of $K$. More specifically, these authors proved the following (see also \cite[Theorem 5.3]{Eugene}):

\begin{theorem}\label{thm: SPR subspace iif K' infinite}
    Suppose that $K$ is a compact Hausdorff space. Then, the following are equivalent:
    \begin{enumerate}[(i)]
        \item $K'$ is infinite.
        \item $c_0$ embeds isometrically into $C(K)$ as an SPR subspace.
        \item $C(K)$ contains an (infinite-dimensional) SPR subspace.
    \end{enumerate}
\end{theorem} 

In light of the previous result, the authors ask whether a \emph{large compact space $K$}, in terms of the smallest ordinal $\alpha$ for which $K^{(\alpha)}$ is non-empty, guarantees the existence of spaces \emph{bigger} than $c_0$ that can be embedded into $C(K)$ in an SPR way. More concretely, and taking into account the fact that $c_0$ is isomorphic to $c=C[1,\omega]$, they raise the following question:

\begin{question}{\cite[Question 6.4]{FOPT}}\label{question:timur's-question-0}
If $K^{(\alpha)}$ is infinite, does $C(K)$ contain an SPR copy of $C[0, \omega^\alpha]$? 
\end{question}

Recall that for any ordinal $\alpha$ we can consider the \emph{ordinal interval} $[0,\alpha]=\{\beta: 0\leq \beta \leq \alpha\}$. It can be shown that these sets, when equipped with the topology generated by the subbase of all sets of the form $[0,\beta_1)$ and $(\beta_2,\alpha]$, for $\beta_1,\beta_2\in [0,\alpha]$, are scattered compact Hausdorff spaces \cite[Corollary 8.6.7]{Semadeni}. Note that the interval $(\beta_1,\beta_2]$ is a clopen set in this topology for every $\beta_1<\beta_2$. Moreover, observe that $[0,\omega^\alpha]$ is homeomorphic to $[1,\omega^\alpha]$ whenever $\alpha\geq 1$, and hence the spaces $C[0,\omega^\alpha]$ and $C[1,\omega^\alpha]$ are lattice isometric. For our purposes, it will be convenient to use the notation $C[1,\omega^\alpha]$, and thus we we will do so throughout this work. Finally, recall that every ordinal $\alpha >0$ can be represented uniquely as
\begin{equation}\label{eq: cantor normal form}
    \alpha = w^{\beta_1}\cdot k_1+\ldots +\omega^{\beta_n}\cdot k_n,
\end{equation}
where $n\geq 1$, $\alpha \geq \beta_1>\ldots >\beta_n\geq 0$ and $k_1,\ldots, k_n$ are non-zero natural numbers. This representation is known as \textit{Cantor's normal form}.

With the aim of better understanding Question \ref{question:timur's-question-0}, we dedicate this article to analyzing the following problem:

\begin{main-question}\label{main-question}
Let $K$ be a compact Hausdorff space and let $\alpha$ be a finite ordinal. If $K^{(\alpha)}$ is non-empty, does $C(K)$ contain a subspace isometric to $C[1,\omega^\alpha]$ which does stable phase retrieval?     
\end{main-question}

We would like to clarify two important aspects regarding this new question. On the one hand, the \emph{isomorphic counterpart} of it has a trivial answer, since $C[1,\omega^\alpha]$ is isomorphic to $c_0$ whenever $1\leq \alpha<\omega$. Second, the \hyperref[main-question]{Main Question} is not a direct adaptation of Question \ref{question:timur's-question-0} to the isometric case: we are relaxing the hypothesis of $K^{(\alpha)}$ being infinite to simply $K^{(\alpha)}\neq\varnothing$.

\subsection{Summary of the results}
As previously mentioned, this article is devoted to studying the \hyperref[main-question]{Main Question}, which is deeply inspired by \Cref{thm: SPR subspace iif K' infinite}. \Cref{sec: two simple observations} contains two observations concerning our question. First, we recall that if a $C(K)$-space contains an isometric copy of $C[1,\omega^\alpha]$, then $K^{(\alpha)}\neq\varnothing$ and this justifies why we must assume (at least) the hypothesis that $K^{(\alpha)}$ is non-empty. Second, we prove that our problem has a negative solution when $\alpha=1$. This was already done in \cite[Theorem 6.1]{FOPT} (\Cref{thm: SPR subspace iif K' infinite} in this paper), but here we provide a generalization of that fact.

Next, in \Cref{sec: simplification}, we refine the techniques of \cite{RS} to prove that if $K^{(\alpha)}$ is non-empty, then we can build a \emph{nice} isometric embedding $T:C[1,\omega^\alpha]\to C(K)$ in the sense that it preserves SPR subspaces (\Cref{prop: implication K_alpha copy of Comegaalpha}). This result reduces the \hyperref[main-question]{Main Question} to studying SPR embeddings between spaces of the form $C[1,\omega^\alpha]$, which is what we will do in \Cref{sec: main results} and will enable us to present a complete answer to our question: 
\begin{itemize}
    \item Given $3\leq \alpha<\omega$, if $K^{(\alpha)}\neq\varnothing$, then $C(K)$ contains a subspace isometric to $C[1,\omega^\alpha]$ doing SPR (\Cref{coro: alpha>=3 Yes}). 
    \item For $\alpha=2$, we distinguish two cases. If $|K''|=1$, then $C[1,\omega^2]$ cannot be embedded isometrically into $C(K)$ in an SPR way (\Cref{prop: Cw2 cannot be SPR embedded into itself}). However, if $|K''|\geq 2$, we can find an isometric SPR embedding of $C[1,\omega^2]$ into $C(K)$ (\Cref{coro: |K''|>=2 Yes}).
\end{itemize}

\section{Some observations on the Main Question}\label{sec: two simple observations}

This section is devoted to stating two \emph{simple observations} regarding the \hyperref[main-question]{Main Question}:
\begin{enumerate}
    \item We start by recalling the necessity of assuming the hypothesis that $K^{(\alpha)}\neq\varnothing$. In particular, we prove that if $C(K)$ contains an isometric copy of $C[1,\omega^\alpha]$, then $K^{(\alpha)}\neq\varnothing$.
    \item Next, we will show that our question has a negative answer for $\alpha=1$. Even though this was already proven in \Cref{thm: SPR subspace iif K' infinite} \cite[Theorem 6.1]{FOPT}, here we will extend this fact to general Banach lattices (not necessarily $C(K)$-spaces).
\end{enumerate}

Recall that if $S$ is a topological space, $S'$ denotes the set of accumulation points (or non-isolated points) of $S$, also known as the \textit{derived set} of $S$. This operation over the set $S$ can be iterated $\alpha$ times for any ordinal $\alpha$, yielding the \textit{$\alpha$-th Cantor--Bendixson derivative of $S$}, denoted $S^{(\alpha)}$. A \textit{perfect set} is a (closed) non-empty set $S$ such that $S'=S$. If a topological space $S$ does not contain any perfect subset, it is called \textit{scattered}, and we denote by $ht(S)$ its \textit{height}, that is, the minimal ordinal $\alpha$ such that $S^{(\alpha)}=\varnothing$. It should be noted that if $K$ is a scattered compact space, then $ht(K)$ is a successor ordinal and $K^{(ht(K)-1)}$ is finite. For instance, if $\alpha$ is a countable ordinal and $m\in \N$, then $ht([1,\omega^\alpha \cdot m])=\alpha +1$ and $|[1,\omega^\alpha \cdot m]^{(\alpha)}|=m$.
\medskip

The first fact that we will need is a consequence of Holszty\'{n}ski's Theorem \cite{Holsztynski}, together with a simple topological observation:

\begin{theorem}\label{thm: copy of Comegaalpha implies Kalpha}
    Assume that the field of scalars is $\R$ or $\C$. Let $L$ be a scattered compact metric space and $K$ a compact Hausdorff space. If $C(L)$ isometrically embeds into $C(K)$, then $|L^{(ht(L)-1)}|\leq |K^{(ht(L)-1)}|$. 
\end{theorem}

In particular, if $C[1,\omega^\alpha]$ isometrically embeds into $C(K)$, then $K^{(\alpha)}$ must be non-empty.

\begin{proof}
    If $C(L)$ isometrically embeds into $C(K)$, by Holszty\'{n}ski's Theorem \cite{Holsztynski} (valid in both the real and the complex setting), there exists a closed subset of $K$ that maps continuously onto $L$. This, in turn, implies that $|L^{(\alpha)}|\leq |K^{(\alpha)}|$, since $L$ is scattered (cf. \cite[Corollary 3.4]{RS}).
\end{proof}
\begin{rem}\label{rem:PNPP}
The reverse implication also holds. Actually, it was proved in \cite[Theorem 4.1, (iv)$\Rightarrow$(i)]{RS} that the embedding $T:C(L)\rightarrow C(K)$ can be assumed to \textit{preserve norms of positive parts} (\emph{PNPP} for short), i.e., $\|(Tf)^+\|=\|f^+\|$ for every $f\in C(L)$, and hence, to be positive. We will elaborate on this fact in the next section.    
\end{rem}

For the second result, we need to recall some definitions. An \textit{atom} of a Banach lattice $X$ is a non-zero positive element $x_0\in X$ such that whenever $0\leq x\leq x_0$, it follows that $x=\lambda x_0$ for some positive scalar $\lambda$. The span of an atom is always a projection band and its corresponding band projection  $P_{x_0}:X\rightarrow X$ is given by $P_{x_0}x=\lambda_{x_0}(x)x_0$, $x\in X$, where $\lambda_{x_0}$ is a real-valued lattice homomorphism. This lattice homomorphism is called the \textit{coordinate functional} associated to $x_0$ (\cite[p. 8]{Lacey}, see also the paragraph preceding Proposition 2.1 in \cite{BGHMT}). The proof of the following result mimics the \emph{gliding hump argument} used in \cite[Theorem 6.1]{FOPT} to prove the \emph{necessary condition}.

\begin{prop}\label{prop: gliding hump}
    Assume that the field of scalars is $\R$ or $\C$. Let $X$ be a Banach lattice that contains an infinite-dimensional SPR subspace. Then, the closed span of the atoms of $X$ has infinite codimension in $X$.
\end{prop}

\begin{proof}
    Let $E\subseteq X$ be an infinite-dimensional SPR subspace, and assume that $X_0$, the closed span of the atoms of $X$, has finite codimension. It follows that $E_0=E\cap X_0$ has infinite dimension. Let $x\in S_{E_0}$ and $\eps>0$. There exists a linear combination of atoms $y=\sum_{k=1}^n \lambda_{x_k}(y)x_k$, where $(x_k)_{k=1}^n$ are atoms and $(\lambda_{x_k})_{k=1}^n$ their corresponding coordinate functionals, such that $\|x-y\|<\eps$. Again, $E_0\cap \intoo{\bigcap_{k=1}^n \ker \lambda_{x_k}}$ must be infinite-dimensional, so in particular we can find a norm-one vector $z$ in this subspace. Note that $z$ is disjoint from $y$ (i.e., supported on different atoms). Hence,
    \[|x|\wedge|z|\leq \bigl(|x-y|+|y|\bigr)\wedge |z|\leq |x-y|\wedge |z|+|y|\wedge|z|=|x-y|\wedge |z|\leq |x-y|,\]
    and $x$ and $z$ form a normalized $\eps$-almost disjoint pair in $E$. Since $\eps$ was arbitrary, \Cref{thm: real stable phase retrieval} (or \Cref{prop: failing complex SPR} in the complex setting), yields a contradiction.
\end{proof}

\begin{rem}\label{rem: case alpha=1}
Suppose that $X=C(K)$, for some compact Hausdorff space $K$, in the above proposition. Observe that an atom in $C(K)$ is precisely a multiple of the characteristic function of a singleton $\{t_0\}$, where $t_0$ is an isolated point of $K$. Therefore, the closed subspace generated by the atoms of $C(K)$ has infinite codimension in $C(K)$ if and only if $K'$ is infinite. Hence, we can use the last proposition to show implication $(iii)\Rightarrow(i)$ in \Cref{thm: SPR subspace iif K' infinite}. In particular, the \hyperref[main-question]{Main Question} has a negative answer for $\alpha=1$, as no infinite-dimensional subspace can be embedded in an SPR way into $c=C[1,\omega]$.
\end{rem}

\section{A simplification of the question}\label{sec: simplification}

\subsection{``Nice'' embeddings that preserve SPR subspaces}

It was established in \Cref{thm: copy of Comegaalpha implies Kalpha} that the existence of an isometric copy of $C[1,\omega^\alpha]$ inside $C(K)$ implies that $K^{(\alpha)}\neq\varnothing$. Conversely, as noted in \Cref{rem:PNPP}, if $K^{(\alpha)}\neq\varnothing$, then we can construct an isometric embedding of $C[1,\omega^\alpha]$ into $C(K)$ with additional properties, namely, being (positive and) PNPP. This is, in principle, a weaker conclusion than what is needed to answer the \hyperref[main-question]{Main Question}. However, the ideas of \cite{RS} can be refined to obtain an embedding with a property that is \emph{very convenient} for our purposes: preserving SPR subspaces (see \Cref{prop: implication K_alpha copy of Comegaalpha}). This will allow us to reduce the \hyperref[main-question]{Main Question} to the study of isometric SPR embeddings between spaces of the form $C[1,\omega^\alpha]$, which are easier to manipulate.

It should be noticed that an elementary example of embeddings that preserve SPR subspaces are precisely lattice embeddings (that is, those that preserve lattice operations). Indeed, suppose that $T:X\to Y$ is a lattice isometric embedding between two Banach lattices $X$ and $Y$, and let $E$ be a subspace of $X$ which does $C$-SPR. Then, for every $x,y\in E$, we have
\begin{align*}
    \frac{1}{\|T\|}\min_{|\lambda|=1}\|Tx-\lambda Ty\|&\leq \min_{|\lambda|=1}\|x-\lambda y\|\leq C\left\| |x|-|y|\right\|\\
    &\leq C \|T^{-1}\|\| T(|x|-|y|)\|=C\|T^{-1}\|\||Tx|-|Ty| \|,
\end{align*}
and hence $T(E)$ does $C\|T\|\|T^{-1}\|$-SPR in $Y$.

Although in general the condition $K^{(\alpha)}\neq\varnothing$ does not guarantee the existence of a lattice isometric embedding of $C[1,\omega^\alpha]$ into $C(K)$ (see \Cref{rem: counterexample non-existence lattice isomorphic embedding}), we will show in \Cref{prop: implication K_alpha copy of Comegaalpha} that it is still possible to construct a \emph{nice} isometric embedding, in the sense that it preserves SPR subspaces. The crucial property of these \emph{nice} embeddings will be property \eqref{property *} from the following proposition.

\begin{prop}\label{prop: Key-propertySPR}
    Let $K,L$ be compact Hausdorff spaces and $T:C(L)\to C(K)$ an operator with the following property: 
    \begin{equation}
        \text{\textit{for every }} t\in L \text{\textit{ there is }}s_t\in K \text{\textit{ such that }}Tf(s_t)=f(t) \text{\textit{ for every }}f\in C(L). \label{property *}\tag{\textasteriskcentered} 
    \end{equation}
    Then, if $E$ is an SPR subspace of $C(L)$, $T(E)$ is an SPR subspace of $C(K)$.
\end{prop}

\begin{proof}
Let us suppose that $E$ is a $C$-SPR subspace of $C(L)$ for some constant $C\geq 1$, that is, for any $f,g\in E$ we have
$$
\min_{|\lambda|=1}\|f-\lambda g\|\leq C\||f|-|g|\|.
$$
Given any $f,g\in C(L)$ and any $t\in L$,
$$
\bigl(|Tf|-|Tg|\bigr)(s_t)=|Tf(s_t)|-|Tg(s_t)|=|f(t)|-|g(t)|=\bigl(|f|-|g|\bigr)(t),
$$
and hence $\||f|-|g|\|\leq \||Tf|-|Tg|\|$. With this in mind, we obtain for any $f,g\in E$ the following:
$$
\min_{|\lambda|=1}\|Tf-\lambda Tg\|\leq \|T\|\min_{|\lambda|=1}\|f-\lambda g\|\leq C \|T\| \||f|-|g|\|\leq C\|T\| \||Tf|-|Tg|\|.
$$
This shows that $T(E)\subseteq C(K)$ does $C\|T\|$-SPR.
\end{proof}

 \begin{rem}\label{rem: property * implies PNPP}

Note that if an operator $T:C(L)\to C(K)$ has the property \eqref{property *} defined above, then $\|Tf\|\geq \|f\|$ for every $f\in C(L)$, so $T$ is an isomorphic embedding. In particular, if $\|T\|=1$, then $T$ is a linear isometry. If we further assume that $T$ is positive, then $T$ must be PNPP. Indeed, on the one hand, we have that for each $f\in C(L)$, $f\leq f^+$, so $(Tf)^+\leq T(f^+)$ and hence $\|(Tf)^+\|\leq \|T(f^+)\|=\|f^+\|$. On the other hand, property \eqref{property *} implies that $(Tf)^+(s_t)=f^+(t)$ for every $f\in C(L)$ and $t\in L$. Therefore,
\[\|(Tf)^+\|\geq \sup_{t\in L}|(Tf)^+(s_t)|=\sup_{t\in L}|f^+(t)|=\|f^+\|.\]

\end{rem}

\begin{rem}
It should be noted that \Cref{prop: Key-propertySPR} could also be stated (and proven in the same way) for locally compact Hausdorff spaces as follows: given locally compact Hausdorff spaces $K$ and $L$, then any operator $T:C_0(L)\to C_0(K)$ satisfying that
\begin{equation}
        \text{\textit{for every }} t\in L \text{\textit{ there is }}s_t\in K \text{\textit{ such that }}Tf(s_t)=f(t) \text{\textit{ for every }}f\in C_0(L) \tag{\textasteriskcentered} 
    \end{equation}
preserves SPR subspaces. As usual, $C_0(K)$ stands for 
$$
\bigl\{f\in C(K)\::\: \forall \varepsilon>0 \:\text{ }\exists \,G\subseteq K \text{ compact space s.t. } |f(t)|<\varepsilon \text{ }\:\forall t\notin G\bigr\}.
$$
\end{rem}

We can adapt the arguments from \cite{RS} to construct embeddings with the property~\eqref{property *}. To do so, we will need the following analogue of \cite[Proposition 3.10]{RS}. Hereby, we will say that an embedding $T:C(L)\rightarrow C(K)$ is \textit{supported by $U\subseteq K$} if $Tf$ is supported by $U$ for every $f\in C(L)$, that is, if $\text{supp}(f)=\overline{\{t\in K\::\: f(t)\neq 0\}}\subseteq U$ for all $f\in C(L)$. Moreover, for an infinite family of compact spaces $\{L_\gamma\}_{\gamma\in \Gamma}$, we will denote by $K_{\{L_\gamma\}_{\gamma\in \Gamma}}$ the one-point compactification of the disjoint union of the spaces $L_\gamma$, with the additional point simply denoted by $\infty$. If for some compact space $L$ we have $L_\gamma=L$ for every $\gamma\in \Gamma$, we will simply write $K_{\Gamma, L}$. It should be noted that for every ordinal $\alpha<\omega_1$, the ordinal interval $[1,\omega^\alpha]$ is homemomorphic to either $K_{\N,[1,\omega^\beta]}$ if $\alpha=\beta+1$, or to $K_{\{[1,\omega^{\beta_n}]\}_{n\in \N}}$ if $\alpha$ is a limit ordinal and $(\beta_n)_{n\in \N}$ is an increasing sequence converging to $\alpha$ (see for instance \cite[Lemma 3.12]{RS}).

\begin{prop}\label{prop: inductive step for emebddings}
    Let $K$ be a compact Hausdorff space, $U$ an open subset of $K$, and $h\in S_{C(K)}$ positive and supported by $U$. For an infinite family of compact spaces $\{L_\gamma\}_{\gamma\in \Gamma}$, let $\{U_\gamma\}_{\gamma\in \Gamma}$ be a family of pairwise disjoint nonempty open subsets of $K$ such that $U_\gamma\subseteq \{s\in K: h(s)=1\}$ for each $\gamma\in \Gamma$, but $\{s\in K: h(s)=1\} \setminus \intoo[3]{\bigcup_{\gamma\in \Gamma}U_\gamma}$ is nonempty. Assume that for each $\gamma\in \Gamma$ there is a positive isometric embedding $T_\gamma:C(L_\gamma)\rightarrow C(K)$ supported by $U_\gamma$ satisfying property \eqref{property *}.

    Then, there exists a positive isometric embedding $T:C(K_{\{L_\gamma\}_{\gamma\in \Gamma}})\to C(K)$ supported by $U$ satisfying property \eqref{property *}.
\end{prop}

\begin{proof}
    Let $g\in C(K_{\{L_\gamma\}_{\gamma\in \Gamma}})$ and let us decompose it as in \cite[Lemma 3.7]{RS}:
    \[g=g(\infty)\uno_{K_{\{L_\gamma\}_{\gamma\in \Gamma}}} + \sum_{\gamma\in \Gamma}g_\gamma,\]
    where $g_\gamma=g|_{L_\gamma}-g(\infty)\chi_{L_\gamma}$ for each $\gamma\in \Gamma$ and $(\|g_\gamma\|)_{\gamma\in \Gamma}\in c_0(\Gamma)$, so that the sum above is well defined (see the comment at the beginning of \cite[Section 3.2]{RS}). Then, we define $T$ as in \cite[Proposition 3.10]{RS} by
    \[Tg=g(\infty)h + \sum_{\gamma\in \Gamma}T_\gamma g_\gamma,\]
    which is well defined because the decomposition of $g$ in such a way is unique. Moreover, since the sets $(U_\gamma)_{\gamma\in \Gamma}$ on which each $T_\gamma g_\gamma$ is supported are pairwise disjoint and $(\|T_\gamma g_\gamma\|)_{\gamma\in \Gamma}=(\|g_\gamma\|)_{\gamma\in \Gamma}\in c_0(\Gamma)$, \cite[Lemma 3.6]{RS} yields that $Tg\in C(K)$. Clearly, $T$ is linear. Let us check that it satisfies property \eqref{property *}. Fix $t\in K_{\{L_\gamma\}_{\gamma\in \Gamma}}$. On the one hand, if $t\in L_\delta$ for some $\delta\in \Gamma$, then using the property \eqref{property *} of $T_\delta$ we can find $s_t\in U_\delta$ such that $T_\delta f(s_t)=f(t)$ for every $f\in C(L_\delta)$. Therefore, using that $T_\gamma g_\gamma$ vanishes on $U_\delta$ for every $\gamma\neq \delta$, $h(s_t)=1$ and $g_\delta=g|_{L_\delta}-g(\infty)\chi_{L_\delta}$, we obtain that
    \[Tg(s_t)=g(\infty)h(s_t) + \sum_{\gamma\in \Gamma}T_\gamma g_\gamma(s_t)= g(\infty) + T_\delta g_\delta(s_t)= g(\infty) +  g_\delta(t)=g(t).\]
    On the other hand, if $t=\infty$, we choose $s_t\in \{s\in K: h(s)=1\} \setminus \intoo[3]{\bigcup_{\gamma\in \Gamma}U_\gamma}$, so that $h(s_t)=1$ but $T_\gamma g_\gamma(s_t)=0$ for every $\gamma\in \Gamma$. It follows that $Tg(s_t)=g(\infty)$. Finally, by \Cref{rem: property * implies PNPP}, $T_\gamma$ is PNPP for every $\gamma\in \Gamma$. Therefore, \cite[Proposition 3.10]{RS} yields that $T$ is a PNPP isometric embedding. In particular, it is positive.
\end{proof}

Now, we are in disposition of proving \Cref{prop: implication K_alpha copy of Comegaalpha}, which will allow us to simplify the \hyperref[main-question]{Main Question} substantially.

\begin{prop}\label{prop: implication K_alpha copy of Comegaalpha}
    Assume that the field of scalars is $\R$ or $\C$. Let $K$ be a compact Hausdorff space, $U\subseteq K$ open, $\alpha<\omega_1$ a countable ordinal, and $m\in \N$. If $|U^{(\alpha)}|\geq m$, then there exists a positive isometric embedding $T:C[1,\omega^\alpha\cdot m]\to C(K)$ supported by $U$ with property \eqref{property *}. 
\end{prop}

\begin{proof}
    The proof is analogous to \cite[Proposition 3.13]{RS}, except for some minor adaptations. Let us first assume that $m=1$ and the field of scalars is $\R$. The base case $\alpha =0$ is simple. Let $h\in S_{C(K)}$ be a positive function supported by $U$. Then, $T:C(\{1\})\equiv \R\rightarrow C(K)$ given by $Ta=ah$ satisfies the desired properties.
    Now, let us assume that $\alpha>0$ and the statement holds for every $\beta <\alpha$. Let us choose $s_0\in U^{(\alpha)}$, an open neighborhood $V$ such that $s_0\in V\subseteq \overline{V}\subseteq U$, and a positive function $h\in S_{C(K)}$ supported by $U$ that takes the value $1$ in $\overline{V}$. Since $s_0\in U^{(\alpha)}\cap V\subseteq V^{(\alpha)}$, it follows that $ht(V)>\alpha$.
    
    If $\alpha=\beta+1$ is a successor ordinal, by \cite[Lemma 3.1(i)]{RS} we can find nonempty open sets $(U_n)_{n\in \N}$ with pairwise disjoint closures such that $\overline{U_n}\subseteq V$ and $ht(U_n)>\beta$ for each $n\in \N$. We can further assume that $s_0\notin U_n$ for every $n\in \N$. By the induction hypothesis, there exist positive isometric embeddings $T_n:C[1,\omega^\beta]\rightarrow C(K)$ supported by $U_n$ with property \eqref{property *}. Therefore, \Cref{prop: inductive step for emebddings} yields a positive isometric embedding $S:C(K_{\N,[1,\omega^\beta]})\rightarrow C(K)$ supported by $U$ with property \eqref{property *}. Since $[1,\omega^\alpha]$ is homeomorphic to $K_{\N,[1,\omega^\beta]}$, the corresponding spaces of continuous functions are lattice isometric via an operator of composition with this homeomorphism. Composing this lattice isometry with the embedding $S$, we obtain a positive isometric embedding $T:C[1,\omega^\alpha]\rightarrow C(K)$ supported by $U$ that also satisfies property \eqref{property *}.
    
    If $\alpha$ is a limit ordinal, we proceed in a similar way, using \cite[Lemma 3.1(ii)]{RS} instead.
    
    Now, if $m>1$, we can find $m$ distinct points $s_1,\ldots, s_m\in U^{(\alpha)}$ and open disjoint neighborhoods $s_i\in U_i\subseteq U$. Since $[1,\omega^\alpha]$ is homeomorphic to $[\omega^\alpha\cdot (i-1)+1,\omega^\alpha\cdot i]$ for each $i=1,\ldots, m$, we can find embeddings $T_i:C[\omega^\alpha\cdot (i-1)+1,\omega^\alpha\cdot i] \rightarrow C(K)$ with the desired properties, and combine them to obtain a global positive isometric embedding $T:C[1,\omega^\alpha \cdot m]\rightarrow C(K)$ supported by $U$ and satisfying property \eqref{property *}.
    
    Finally, let us extend this proof from real to \textbf{complex scalars}. We have shown that there is a positive isometric embedding $T:C([1,\omega^\alpha];\R)\to C(K;\R)$ supported by $U$ with the property \eqref{property *}. We claim that $T_{\mathbb{C}}:C([1,\omega^\alpha];\mathbb{C})\to C(K;\mathbb{C})$ defined by $T_\mathbb{C}(f+ig)=Tf+iTg$, for $f,g\in C([1,\omega^\alpha];\R)$ is a positive isometric $\mathbb{C}$-linear embedding with the desired properties. Recall that since $T$ is positive, we have $\|T_\mathbb{C}\|=\|T\|=1$ (see, for instance, \cite[Corollary 3.23]{AA-book}). It is clear that $T_\mathbb{C}$ satisfies the property \eqref{property *} and from this we deduce that
\begin{align*}
\sup_{s\in K} |T_\mathbb{C}(f+ig)|(s)&=\sup_{s\in K} \sqrt{|Tf(s)|^2+|Tg(s)|^2}\geq \sup_{t\in [1,\omega^\alpha]} \sqrt{|Tf(s_t)|^2+|Tg(s_t)|^2}    \\
&=\sup_{t\in [1,\omega^\alpha]} \sqrt{|f(t)|^2+|g(t)|^2}=\sup_{t\in[1,\omega^\alpha]}|f+ig|(t)=\|f+ig\|,
\end{align*}
and this shows that $T_\mathbb{C}$ is norm-preserving.
\end{proof}

\begin{rem}\label{rem: counterexample non-existence lattice isomorphic embedding}
    It should be noticed that the assumption $K^{(\alpha)}\neq \varnothing$ does not guarantee the existence of a lattice isomorphic (not necessarily isometric) embedding of $C[1,\omega^\alpha]$ into $C(K)$. 
    
    Indeed, suppose that there exists a lattice isomorphic embedding $T:c\to C[0,1]$. Let us write $f_n=Te_n$ for every natural number $n$. Observe that, as $T$ preserves lattice operations, $(f_n)_{n=1}^\infty$ is a pairwise disjoint positive sequence in $C[0,1]$. Moreover, 
    $$
    \frac{1}{\|T^{-1}\|}\leq \|f_n\|=\|Te_n\|\leq \|T\|, \quad \text{ for every } n\geq 1.
    $$
 Hence, for every $n$ there is $t_n\in[0,1]$ such that $f_n(t_n)\geq \frac{1}{\|T^{-1}\|}$. Passing to a subsequence, we may assume that $t_n\to t_0$ in $[0,1]$. It is clear that $f_n(t_0)=0$ for every $n$ since the $f_n$'s are pairwise disjoint. Moreover, as $T\uno(t_n)\geq Te_n(t_n)=f_n(t_n)\geq \frac{1}{\|T^{-1}\|}$, it follows by the continuity of $T\uno$ that $T\uno(t_0)\geq\frac{1}{\|T^{-1}\|}$. 
 
 Now, let $I$ be an open interval of $[0,1]$ such that $t_0\in I$ and $T\uno(t)>\frac{1}{2\|T^{-1}\|}$ for every $t\in I$, and let $N$ be such that $t_N\in I$. Observe that we have
 $$
 0<T\uno(t)=Te_N(t)+T(\uno-e_N)(t), \quad \text{ for every } t\in I,
 $$
 so $I=A_1\cup A_2$, where $A_1=\left\{t\in I\::\:Te_N(t)>0\right\}$ and $A_2=\left\{t\in I\::\:T(\uno-e_N)(t)>0\right\}$ are open subsets of $I$. Observe that $t_N\in A_1$ and $t_0\in A_2$ and, in addition, $A_1$ and $A_2$ are disjoint as $Te_N$ and $T(\uno-e_N)$ are (recall that $T$ preserves disjoint elements of $c$). But this is impossible, since then the interval $I$ would not be connected. A more general result can be found in \cite[Theorem C]{AGRT}.
\end{rem}

The previous proposition allows us to add one more condition to \cite[Theorem 4.1]{RS} when $L$ is a scattered compact metric space:

\begin{theorem}\label{thm: equivalence embedding CK}
    Let $L$ be a scattered compact metric space and $K$ a compact space. Then, $C(K)$ contains an isometric copy of $C(L)$ if and only if $C(K)$ contains a positive isometric copy of $C(L)$ with property \eqref{property *}. 
\end{theorem}
\begin{proof}
    The reverse implication is trivial. For the forward implication, it suffices to recall that every scattered compact metric space is countable, so it must be homeomorphic to some $[1,\omega^\alpha\cdot m]$, with $\alpha =ht(L)-1$ and $m=|L^{(\alpha)}|$. Then, by \cite[Theorem 4.1]{RS} we know that $C(K)$ contains an isometric copy of $C(L)$ if and only if $|L^{(\alpha)}|=m\leq |K^{(\alpha)}|$. Applying \Cref{prop: implication K_alpha copy of Comegaalpha} with $U=K$ we obtain the result.
\end{proof}

\subsection{Freeman-Oikhberg-Pineau-Taylor's proof revisited}
A careful look at the proof of \cite[Theorem 6.1]{FOPT} (\Cref{thm: SPR subspace iif K' infinite} here) reveals that the same ideas the authors use to construct an SPR copy of $c_0$ inside every $C(K)$-space with $K'$ infinite can also be applied to show that $C_0[1,\omega^2)=\{f\in C[1,\omega^2]\::\: f(\omega^2)=0\}$ contains an isometric copy of $c_0$ doing SPR. For completeness, we include below a detailed proof of this fact.

Throughout the proof, $\uno_{m_1}$ will represent the characteristic function of the clopen set $(\omega\cdot(m_1-1),\omega\cdot m_1]$ for any natural number $m_1\geq 1$. Analogously, given natural numbers $m_1,m_2\geq 1$, $\uno_{m_1,m_2}$ will stand for the characteristic of the singleton $\{\omega\cdot (m_1-1)+m_2\}$.

\begin{prop}\label{prop: SPR copy of c0 in C0[0 w2]}
    There exists an isometric SPR embedding of $c_0$ into $C_0[1,\omega^2)$.
\end{prop}

\begin{proof}
For every $n\in \N$, we define the continuous function on $[1,\omega^2]$ given by the expression
\[x^{(n)}=\uno_{1,n}+\frac{1}{2}\uno_{n+1}+\frac{1}{2}\sum_{m=2}^{n}\uno_{m,n}\]
(see \Cref{Fig: real SPR copy of c0}). It is clear that $x^{(n)}(\omega^2)=0$, as it is supported by $[1,\omega \cdot (n+1)]$.

\begin{figure}[!hbt]
\centering
\includegraphics[scale=1.8]{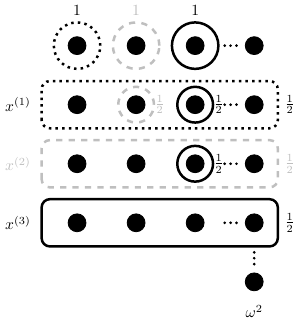}
\caption{\label{Fig: real SPR copy of c0} Representation of $x^{(n)}$. The points in $[1,\omega^2]$ are ordered from left to right and from the top to the bottom.}
\end{figure}

The sequence $(x^{(n)})_n$ spans an isometric copy of $c_0$. Indeed, given $(\alpha_n)_n\in c_{00}$ with $\bigvee_n|\alpha_n|=1$, it follows that $x=\sum_n\alpha_n x^{(n)}$ takes the value $\alpha_n$ at the point $n$ for every $n\in \N$, so $\|x\|\geq 1$. On the other hand, observe that $x$ vanishes outside of $[1,\omega\cdot (k+1)]$ for $k$ big enough. Moreover, the function $x$ takes the values $\alpha_n$ at the point $n$, $0$ at $\omega$, $\frac{1}{2}\alpha_n$ at $[\omega\cdot n+1, \omega\cdot n+n]\cup\{\omega\cdot (n+1)\}$, and $\frac{1}{2}(\alpha_n+\alpha_m)$ at $\omega \cdot n+m$ for every $m>n$ and $n\in \N$. In all the cases above, $|x|\leq 1$, so we can conclude that $\|x\|=1$.

Now, let us check that $E=\overline{\spn}\{x^{(n)}:n\in \N\}$ does real SPR. By \Cref{thm: real stable phase retrieval}, it suffices to show that $\||x|\wedge |y|||\geq \frac{1}{3}$ for every $x,y\in S_E$. Put $x=\sum_n\alpha_n x^{(n)}$ and $y=\sum_n\beta_n x^{(n)}$ for some $(\alpha_n)_n, (\beta_n)_n\in S_{c_0}$. Find $n,m\in \N$ such that $|\alpha_n|=1=|\beta_m|$. If $n=m$, then $|x|\wedge|y|(n)=1$. Otherwise, assume that $n<m$. If $|\alpha_m|\geq \frac{1}{3}$, then $|x|\wedge|y|(m) =|\alpha_m|\wedge |\beta_m|\geq \frac{1}{3}$, and the same can be done when $|\beta_n|\geq \frac{1}{3}$. Finally, if $|\alpha_m|, |\beta_n|<\frac{1}{3}$, it follows that $|\alpha_n+\alpha_m|, |\beta_n+\beta_m|\geq \frac{2}{3}$, so $|x|\wedge|y|(\omega \cdot n + m) \geq \frac{1}{3}$. Observe that by \Cref{thm: real stable phase retrieval} we can conclude that $E$ does 3-SPR inside $C(K)$.
\end{proof}

\Cref{prop: implication K_alpha copy of Comegaalpha}, together with the previous lemma, now yields the \emph{sufficiency} part of \cite[Theorem 6.1]{FOPT} (implication $(i)\Rightarrow (ii)$ in \Cref{thm: SPR subspace iif K' infinite}).

\begin{corr}
Let $K$ be a compact Hausdorff space. If $K'$ is infinite, then $C(K)$ contains an isometric copy of $c_0$ which does SPR.    
\end{corr}

\begin{proof}
Suppose that $K'$ is infinite, which is equivalent to $K''$ being non-empty. By \Cref{prop: implication K_alpha copy of Comegaalpha}, there exists an isometric embedding $T:C[1,\omega^2]\to C(K)$ which preserves SPR subspaces. By \Cref{prop: SPR copy of c0 in C0[0 w2]}, $C_0[1,\omega^2)$ contains a subspace $E$ isometric to $c_0$ doing SPR. Since, $C_0[1,\omega^2)$ is a sublattice of $C[1,\omega^2]$, then $E$ also does SPR in $C[1,\omega^2]$. Therefore, $T(E)$ is an SPR subspace of $C(K)$ isometric to $c_0$.
\end{proof}

\section{Main results}\label{sec: main results}

In this section, we will provide a \emph{complete answer} to the \hyperref[main-question]{Main Question}: does the condition $K^{(\alpha)}\neq\varnothing$ (for $\alpha<\omega$) imply the existence of an isometric copy of $C[1,\omega^\alpha]$ inside $C(K)$ doing SPR? Note that it is trivially true for $\alpha=0$, while for $\alpha=1$ we already know that it is false (\Cref{rem: case alpha=1}); thus, we only need to study the cases where $2\leq \alpha<\omega$. Due to \Cref{prop: implication K_alpha copy of Comegaalpha}, this question is essentially reduced to the study of SPR embeddings between spaces of the form $C[1,\omega^\alpha]$, which is our next step.

 We start by proving the following lemma. Note that the fact that $C[1,\omega^\alpha]$ and $c_0$ are isomorphic for any $1\leq \alpha<\omega$ already follows from \cite[Theorem 1]{BP}, but the estimate for the isomorphism constant can be improved.
    
 \begin{lemma}
 Given $1\leq \alpha<\omega$, $C[1,\omega^\alpha]$ is linearly $2(\alpha+1)$-isomorphic to $c_0([1,\omega^\alpha])=c_0$.
 \end{lemma}
\begin{proof}
    Let us establish some notation. Given an ordinal $1\leq t<\omega^\alpha$, we can express it using Cantor's normal form \eqref{eq: cantor normal form} as
    \[t=\sum_{0\leq \beta<\alpha} \omega^\beta\cdot k_\beta, \quad \text{where } (k_\beta)_\beta\in (\N\cup\{0\})^\alpha\setminus \{0\}^\alpha.\]
    We define $\beta_0(t)=\min\{0\leq \beta<\alpha: k_\beta\neq 0\}$ (note that for every $t$ this set of ordinals is non-empty, as $(k_\beta)_\beta$ is non-zero, so by the well-ordering principle of ordinal numbers it has a minimum), and $\ell:[1,\omega^\alpha]\rightarrow [1,\omega^\alpha]$ by
    \[\ell(t)=\left\{
    \begin{array}{ll}
        \sum_{\beta_0(t)+1< \beta<\alpha} \omega^\beta\cdot k_\beta + \omega^{\beta_0(t)+1}\cdot (k_{\beta_0(t)+1}+1) & \text{ if }1\leq t<\omega^\alpha,\\
        \omega^\alpha &  \text{ if } t=\omega^\alpha
    \end{array}
    \right.\]
    (it is understood that $k_\alpha=0$ for $t<\omega^\alpha$). Note that for every $1\leq t<\omega^\alpha$, $t$ belongs to the $\beta_0(t)$-derivative of $[1,\omega^\alpha]$, but not to the $(\beta_0(t)+1)$-derivative, and $\ell(t)$ belongs to the $(\beta_0(t)+1)$-derivative, but not to the $(\beta_0(t)+2)$-derivative. Next, we define the following operator
    \[\fullfunction{R}{C[1,\omega^\alpha]}{c_0([1,\omega^\alpha])}{f}{Rf=(a(t))_t,}\]
    where $a(\omega^\alpha)=\frac{1}{2}f(\omega^\alpha)$ and $a(t)=\frac{1}{2}\bigl(f(t)-f(\ell(t))\bigr)$ for every $1\leq t<\omega^\alpha$. Clearly $R$ is well-defined and contractive. Moreover, for every $f\in S_{C[1,\omega^\alpha]}$, $\|Rf\|\geq \frac{1}{2(\alpha+1)}$. Indeed, for every $1\leq t<\omega^\alpha$, we have that $\ell^{\alpha-\beta_0(t)}(t)=\omega^\alpha$ (where $\ell^{\alpha-\beta_0(t)}$ represents the composition of $\ell$ with itself $\alpha-\beta_0(t)$ times), and we can write
    \[f(t)=2\sum_{j=0}^{\alpha-\beta_0(t)} a(\ell^j(t))\]
    for every $f\in C[1,\omega^\alpha]$. Therefore, if there were some $f\in S_{C[1,\omega^\alpha]}$ such that $\|Rf\|< \frac{1}{2(\alpha+1)}$, there would exist some $t\in [1,\omega^\alpha]$ with $|f(t)|=1$, so
    \[1=|f(t)|\leq 2 \sum_{j=0}^{\alpha-\beta_0(t)} |a(\ell^j(t))| < \frac{\alpha-\beta_0(t)+1}{\alpha +1}\leq 1,\]
    which leads to a contradiction. Hence, the proof of the claim is concluded.
\end{proof}

Next, we exploit this isomorphism, together with the SPR embedding of $c_0$ into $C_0[1,\omega^2)$ from \Cref{prop: SPR copy of c0 in C0[0 w2]}, to show the following:

\begin{prop}\label{prop: SPR embedding of Cwn into Cwncupw2}
    Given a finite ordinal $\alpha\geq 2$, there exists an isometric SPR embedding of $C[1,\omega^\alpha]$ into $ C_0[1,\omega^2)\oplus_\infty C[1,\omega^\alpha]$.
\end{prop}

\begin{proof}
  Fix any $2\leq \alpha<\omega$. By the previous lemma, there exists a $2(\alpha+1)$-isomorphism $R:C[1,\omega^\alpha]\to c_0$. Replacing $R$ by $\frac{R}{\|R\|}$ and $R^{-1}$ by $\|R\|R^{-1}$ if necessary, we may assume that $\|R\|\leq 1$. Thus, $\|R^{-1}\|\leq 2(\alpha+1)$. Now, consider the isometric $3$-SPR embedding $S:c_0\rightarrow C_0[1,\omega^2)$ constructed in \Cref{prop: SPR copy of c0 in C0[0 w2]}, and define the operator
    \[\fullfunction{T}{C[1,\omega^\alpha]}{C_0[1,\omega^2)\oplus_\infty C[1,\omega^\alpha]}{f}{Tf=(S\circ Rf,f).}\]
    Clearly, $T$ is an isometric embedding, as it is contractive and $\|Tf\|\geq \|f\|$. To conclude, we will use \Cref{thm: real stable phase retrieval} to show that $T(C[1,\omega^\alpha])$ does $6(\alpha+1)$-SPR in $C_0[1,\omega^2)\oplus_\infty C[1,\omega^\alpha]$. Given $f,g\in S_{C[1,\omega^\alpha]}$, then $\|Rf\|,\|Rg\|\geq \frac{1}{\|R^{-1}\|}$, so we have
    $$
        \left\||Tf|\land |Tg| \right\|_\infty\geq \left\||SRf|\land |SRg| \right\|_\infty\geq\frac{1}{\|R^{-1}\|}\left\|\left|S\left(\frac{Rf}{\|Rf\|}\right)\right|\land \left|S\left(\frac{Rg}{\|Rg\|}\right)\right| \right\|_\infty 
        \geq  \frac{1}{6(\alpha+1)}.
   $$ 
\end{proof}

The topological properties of the ordinal intervals yield the following corollaries:

\begin{corr}\label{coro: SPR embedding of Cwalpha into Cwalpha}
    Given a finite ordinal $\alpha\geq 3$, there exists an isometric SPR embedding of $C[1,\omega^\alpha]$ into $C[1,\omega^\alpha]$.
\end{corr}
\begin{proof}
    It suffices to observe that the following spaces are lattice isometric (here, $\sqcup$ denotes a disjoint union):
    \[C[1,\omega^2]\oplus_\infty C[1,\omega^\alpha] =C([1,\omega^2]\sqcup[1,\omega^\alpha])=C([1,\omega^2]\cup(\omega^2,\omega^\alpha])=C[1,\omega^\alpha].\qedhere\]
\end{proof}

\begin{corr}\label{coro: SPR embedding of Cw2 into C2w2}
    There exists an isometric SPR embedding of $C[1,\omega^2]$ into $C[1,\omega^2\cdot~2]$.
\end{corr}
\begin{proof}
    Just observe that
    \[C_0[1,\omega^2)\oplus_\infty C[1,\omega^2]\subseteq C[1,\omega^2]\oplus_\infty C[1,\omega^2]=C[1,\omega^2\cdot 2].\qedhere\]
\end{proof}

Finally, using \Cref{prop: implication K_alpha copy of Comegaalpha} we obtain the main results of the paper:

\begin{theorem}\label{coro: alpha>=3 Yes}
Let $K$ be a compact Hausdorff space with $K^{(\alpha)}\neq \varnothing$ for some $3\leq\alpha<\omega$. Then there is an isometric SPR embedding of $C[1,\omega^\alpha]$ into $C(K)$.  
\end{theorem}
\begin{proof}
Suppose that $K^{(\alpha)}\neq \varnothing$ for some $3\leq\alpha<\omega$. By \Cref{prop: implication K_alpha copy of Comegaalpha}, there exists an isometric embedding $T:C[1,\omega^\alpha]\to C(K)$ which preserves the SPR subspaces of $C[1,\omega^\alpha]$. Since, by \Cref{coro: SPR embedding of Cwalpha into Cwalpha}, $C[1,\omega^\alpha]$ embeds isometrically into itself in an SPR way, we infer that $C[1,\omega^\alpha]$ embeds isometrically into $C(K)$ in an SPR way.    
\end{proof}

\begin{theorem}\label{coro: |K''|>=2 Yes}
If $K$ is a compact Hausdorff space with $|K''|\geq 2$, then there is an isometric SPR embedding $T: C[1,\omega^2]\rightarrow C(K)$.
\end{theorem}
\begin{proof}

By \Cref{prop: implication K_alpha copy of Comegaalpha} we know that if $|K''|\geq 2$, there exists a positive linear isometric embedding of $C[1,\omega^2\cdot 2]$ into $C(K)$ with the property~\eqref{property *}. In particular, $T$ preserves the SPR subspaces of $C[1,\omega^2\cdot 2]$. Thus, we have that $C[1,\omega^2]$ can be embedded isometrically in an SPR way into any $C(K)$-space with $|K''|\geq 2$.    
\end{proof}
 
In view of the preceding corollaries, one might wonder whether it is possible to embed $C[1,\omega^2]$ into itself in an SPR way, as this is the case for $C[1,\omega^\alpha]$ when $\alpha>2$. We will see next that this is not possible.

\begin{prop}\label{prop: Cw2 cannot be SPR embedded into itself}
    Assume that the field of scalars is $\R$ or $\C$. $C[1,\omega^2]$ cannot be embedded isometrically into a $C(K)$-space with $|K''|=1$ in an SPR way.
\end{prop}

\begin{proof}
Let us first establish the notation: we will denote by $\uno$ the constant function on $[1,\omega^2]$. Given a natural number $m_1\geq 1$, $\uno_{m_1}$ will represent the characteristic function of the clopen set $(\omega\cdot(m_1-1),\omega\cdot m_1]$.
Similarly, given natural numbers $m_1,m_2\geq 1$, $\uno_{m_1,m_2}$ will stand for the characteristic function of the clopen set $\{\omega\cdot (m_1-1)+m_2\}$.

Now, given an isometric embedding $T:C[1,\omega^2]\to C(K)$ and natural numbers $m_1,m_2\geq 1$, there must exist $t_{m_1,m_2}\in K$ such that
\begin{equation}\label{eq:key-prop21}  \bigl|\bigl(T\uno+T\uno_{m_1}+T\uno_{m_1,m_2} \bigr)(t_{m_1,m_2}) \bigr|=3.  
\end{equation}
Indeed, recall that $\uno+\uno_{m_1}+\uno_{m_1,m_2}$ has norm $3$ in $C[1,\omega^2]$ and $T$ is norm-preserving. From the expression \eqref{eq:key-prop21}, given that $\|T\uno\|=\|T\uno_{m_1}\|=\|T\uno_{m_1,m_2}\|=1$, we deduce the existence of some unimodular scalar $\theta_{m_1,m_2}$ such that
\[T\uno(t_{m_1,m_2})=T\uno_{m_1}(t_{m_1,m_2})=T\uno_{m_1,m_2}(t_{m_1,m_2})=\theta_{m_1,m_2}.\]

We claim that there exists $t_0\in \{t_{m_1,m_2}\::\: m_1,m_2\geq 1\}''\subseteq K''$.  To prove it, first we show that if $m_1\geq 1$ is fixed, then the points $t_{m_1,m_2}$ with $m_2\geq 1$ are distinct. Indeed, suppose that $t_{m_1,m_2}= t_{m_1,m_2'}$ for some $m_2\neq m_2'$. Then, $\theta_{m_1,m_2}^{-1}\uno_{m_1,m_2}+ \theta_{m_1,m_2'}^{-1}\uno_{m_1,m_2'}$ has norm $1$ in $C[1,\omega^2]$, but $T(\theta_{m_1,m_2}^{-1}\uno_{m_1,m_2}+\theta_{m_1,m_2'}^{-1}\uno_{m_1,m_2'})(t_{m_1,m_2})=2$, which is a contradiction with the fact that $T$ is norm-preserving. Therefore, for each $m_1\geq 1$ the sequence $(t_{m_1,m_2})_{m_2=1}^\infty$ must have at least an accumulation point $t_{m_1}\in K'$. By the continuity of each $\uno_{m_1}$, it follows that $|T\uno (t_{m_1})|=|T\uno_{m_1}(t_{m_1})|=1$. Repeating the same argument, we can show that the points $t_{m_1} $ with $ m_1\geq 1$ are distinct, so there must be an accumulation point $t_0\in \{t_{m_1}\::\: m_1\geq 1\}'\subseteq (K')'=K''$. Since $|K''|=1$, it follows that $K''=\{t_0\}$. Observe that, since $|T\uno(t_{m_1})|=1$ for all $m_1$, we have $|T\uno(t_0)|=1$. In addition, $T\uno_1(t_0)=0$. Otherwise, we could find a unimodular scalar $\lambda$ such that $\lambda T\uno_1(t_0)>0$. Since $\uno_1$ and $\uno_{m_1}$ are disjoint for any $m_1>1$, we have that $\|\lambda \uno_1+\nu \uno_{m_1}\|=1$ for any choice of unimodular $\nu$. In particular, putting $\nu=(T\uno_{m_1}(t_{m_1}))^{-1}$, we get that
\[1\geq |\lambda T\uno_1+\nu T\uno_{m_1}|(t_{m_1})=|\lambda T\uno_1(t_{m_1})+1|.\]
Taking the limit in $m_1$ we obtain that
\[1\geq |\lambda T\uno_1(t_0)+1|=\lambda T\uno_1(t_0)+1>1,\]
which is a contradiction. Therefore, $T\uno_1(t_0)=0$. 

Now, suppose that $T(C[1,\omega^2])$ does SPR. By \Cref{thm: real stable phase retrieval} (or \cite[Proposition 2.5]{CGH} in the complex setting), there exists a constant $C>0$ such that for every $n,m\geq 1$, with $n\neq m$, we have
$$
\||T\uno_{1,n}|\land |T\uno_{1,m}| \|_\infty \geq \frac{1}{C}.
$$
Thus, for every $n\geq 1$, there exists a sequence $(s_{n,m})_{m=n+1}^\infty\subseteq K$ such that 
$$
(|T\uno_{1,n}|\land |T\uno_{1,m}|)(s_{n,m})\geq \frac{1}{C}.
$$
Fix any $n\geq 1$. We claim that the sequence $(s_{n,m})_{m=n+1}^\infty$ has infinitely many distinct elements. Even more: no term of the sequence $(s_{n,m})_{m=n+1}^\infty$ can be repeated an infinite number of times. Suppose that a point $s\in K$ appears $N$ times in the sequence, that is, there are $m_1<m_2<\cdots<m_N$ such that $s_{n,m_1}=s_{n,m_2}=\cdots=s_{n,m_N}=s$. Since $|T\uno_{1,m_k}  (s_{n,m_k})|\geq\frac{1}{C}$ for $k=1,\ldots, N$, we can find a scalar $\theta_{k}$ of modulus $1$ such that
$$
\theta_{k}\,T\uno_{1,m_k}(s_{n,m_k})\geq \frac{1}{C}.
$$
It is clear that $\sum_{k=1}^N\theta_{k}\,\uno_{1,m_k}$ has norm $1$ in $C[1,\omega^2]$, and since $T$ is norm-preserving, the function
$\sum_{k=1}^N\theta_{k}\,T\uno_{1,m_k}$ must have norm $1$. If we evaluate this function at $s$ we have
$$
1\geq\sum_{k=1}^N\theta_{k}\,T\uno_{1,m_k}(s_{n,m_k})\geq \frac{N}{C},
$$
so $N$ must be less than or equal to $C$. Therefore, for any $n\geq 1$ we can find an accumulation point $s_{n}$ of the sequence $(s_{n,m})_{m=n+1}^\infty$, which clearly belongs to $K'$.

By the continuity of each $T\uno_{1,n}$, we have that $|T\uno_{1,n}(s_{n})|\geq \frac{1}{C}$ for every $n\geq 1$. In addition, using the same argument, it is easy to check that for every $n\geq 1$, the sequence $(s_{n})_{n=1}^\infty$ has infinitely many different points. In particular, this implies that the sequence $(s_{n})_{n=1}^\infty\subseteq K'$ has an accumulation point, which must be $t_0$, as $K''=\{t_0\}$.

Now we consider the family of functions $\uno-\uno_1+\zeta_m\uno_{1,m}$, where $m\geq 1$ and $\zeta_m$ are unimodular scalars. Observe that since these functions have norm $1$ in $C[1,\omega^2]$, then $\|T\uno-T\uno_1+\zeta_mT\uno_{1,m}\|_\infty=1$ for every $m\geq 1$. In the first part of the proof, we have pointed out that $T\uno(t_0)=\theta$ for some scalar $\theta$ of modulus $1$ and $T\uno_1(t_0)=0$. Fix $0<\varepsilon<\frac{1}{2C}$ and let $U$ be a neighborhood of $t_0$ in $K$ such that 
$$
|T\uno(t)-\theta|<\varepsilon \quad \text{ and } \quad |T\uno_1(t)|<\varepsilon \quad \text{ for every } t\in U.
$$
Since $t_0$ is an accumulation point of the sequence $(s_{n})_{n=1}^\infty$, there exists $M\in\N$ such that $s_{M}\in U$. We also know that $|T\uno_{1,M}(s_{M})|\geq \frac{1}{C}$, so take $\xi_{M}$ of modulus 1 such that 
$$
\xi_{M}\,T\uno_{1,M}(s_{M})\geq \frac{1}{C}.
$$
Finally, we consider the function $T\bigl(\uno-\uno_1+\theta \,\xi_{M}\,\uno_{1,M}\bigr)$, which should have norm $1$ in $C(K)$. However, if we evaluate this function at $s_{M}$ we obtain the following:
\begin{eqnarray*}
 |T\uno-T\uno_1+ \theta \,\xi_{M}\,T\uno_{1,M} |(s_{M}) &= &|T\uno(s_{M})-\theta +\theta + \theta \,\xi_{M}\,T\uno_{1,M}(s_{M}) - T\uno_1(s_{M})| \\
 & \geq& |\theta + \theta \,\xi_{M}\,T\uno_{1,M}(s_{M})|- |T\uno(s_{M})-\theta| - |T\uno_1(s_{M})| \\
 &=& 1+\xi_{M}\,T\uno_{1,M}(s_{M}) - 2\varepsilon\geq 1+\frac{1}{C}-2\varepsilon>1,
\end{eqnarray*}
which is a contradiction.
\end{proof}

Since the second derivative of $[1,\omega^2]$ is the singleton $\{\omega^2\}$, the above proposition has the following immediate consequence:

\begin{corr}\label{coro: Cw2 cannot be SPR embedded into itself}
  $C[1,\omega^2]$ cannot be embedded isometrically into itself in an SPR way.   
\end{corr}

\Cref{prop: Cw2 cannot be SPR embedded into itself} can be generalized to other finite ordinals $\alpha$ using a similar argument, even though the consequences are not as interesting for our purposes as in the case $\alpha=2$.

\begin{prop}
    Let $2\leq \alpha<\omega$ be an ordinal and $K$ be a compact Hausdorff space with $|K^{(\alpha)}|=1$. If $T:C[1,\omega^\alpha]\rightarrow C(K)$ is an isometric embedding, then for every $\eps>0$ there exists $f_1,\ldots, f_\alpha\in S_{C[1,\omega^\alpha]}$ such that
    \[\norm[3]{\bigwedge_{\beta=1}^\alpha |Tf_\beta|}<\eps.\]
\end{prop}

To summarize, we have just seen that $C[1,\omega^2]$ does not embed isometrically into itself in an SPR way (\Cref{prop: Cw2 cannot be SPR embedded into itself}), but we already knew that $c_0$ does so in $C[1,\omega^2]$ (in fact, in $C_0[1,\omega^2)$ by \Cref{prop: SPR copy of c0 in C0[0 w2]}). How optimal is the latter? We will show next that we can slightly improve this, as $C_0[1,\omega^2)$ embeds isometrically into itself doing SPR.

\begin{prop}\label{prop: SPR embedding of C0w2 into itself}
    There exists an isometric SPR embedding of $C_0[1,\omega^2)$ into itself.
\end{prop}

\begin{proof}
    First, observe that $C_0[1,\omega^2)$ is lattice isometric to $C_0[1,\omega^2)\oplus_\infty C_0[1,\omega^2)$. Next, recall that $C_0[1,\omega^2)$ is linearly 4-isomorphic to $c_0([1,\omega^2))=c_0$ by means of the contractive operator $R$ that sends any $f\in C_0[1,\omega^2)$ to the collection of coefficients $a(\omega\cdot n)=\frac{1}{2}f(\omega\cdot n)$ and $a(\omega\cdot (n-1)+m)=\frac{1}{2}(f(\omega\cdot (n-1)+m)-f(\omega\cdot n))$, $n,m\geq 1$ (in this case, $\|R^{-1}\|\leq 4$). Let $S:c_0\rightarrow C_0[1,\omega^2)$ be the isometric SPR embedding constructed in \Cref{prop: SPR copy of c0 in C0[0 w2]}. Then, the operator
    \[\fullfunction{T}{C_0[1,\omega^2)}{C_0[1,\omega^2)\oplus_\infty C_0[1,\omega^2)}{f}{Tf=(S\circ Rf,f)}\]
    is an isometric SPR embedding of $C_0[1,\omega^2)$ into itself. Indeed, $T$ is clearly contractive and $\|Tf\|\geq \|f\|$, since one of the factors is the identity operator. Now we check that the image of $T$ does SPR in the same way as we did in \Cref{prop: SPR embedding of Cwn into Cwncupw2}.

 We observe that if $f,g\in S_{C_0[1,\omega^2)}$, then $\|Rf\|,\|Rg\|\geq \frac{1}{\|R^{-1}\|}$, so since $S$ is a $3$-SPR embedding, then
$$
        \left\||Tf|\land |Tg| \right\|_\infty\geq \left\||SRf|\land |SRg| \right\|_\infty\geq\frac{1}{\|R^{-1}\|}\left\|\left|S\left(\frac{Rf}{\|Rf\|}\right)\right|\land \left|S\left(\frac{Rg}{\|Rg\|}\right)\right| \right\|_\infty 
        \geq  \frac{1}{12}.
   $$ 
        Thus, \Cref{thm: real stable phase retrieval} yields that $T$ does $12$-SPR.   
\end{proof}
\begin{rem}
This proposition shows that assertion $(ii)$ in \Cref{thm: SPR subspace iif K' infinite} can be replaced by: \emph{$C_0[1,\omega^2)$ embeds isometrically into $C(K)$ as an SPR subspace}.   
\end{rem}

\subsection{Open questions}
As a consequence of the results of this section we can provide a fairly complete answer to the \hyperref[main-question]{Main Question}  for every finite ordinal $\alpha$. Observe that for $\alpha>2$ the conclusions turn out to be \emph{stronger} than what was originally conjectured in \Cref{question:timur's-question-0}. 
However, the SPR constant of these subspaces will diverge as $\alpha$ grows. This might be an artifact of the proof, since we are trying to build all the overlapping needed to avoid almost disjoint pairs of functions of $C[1,\omega^\alpha]$ in such a small place as $[1,\omega^2]$. 

\begin{question}\label{quest: SPR constant explodes}
\Cref{prop: SPR embedding of Cwn into Cwncupw2} actually shows that there is a $6(\alpha+1)$-SPR isometric embedding of $C[1,\omega^\alpha]$ into $C_0[1,\omega^2)\oplus_\infty C[1,\omega^\alpha]$. Since $C_0[1,\omega^2)\oplus_\infty C[1,\omega^\alpha]$ is lattice isometric to $C[1,\omega^\alpha]$ for $\alpha\geq 3$ (see the proof of \Cref{coro: SPR embedding of Cwalpha into Cwalpha}), this implies that in these cases $C[1,\omega^\alpha]$ isometrically embeds into itself doing SPR with constant $\leq 6(\alpha+1)$. 
   Can the latter be improved so that the SPR constant is uniform in $\alpha$?
\end{question}

Another limitation of the proof of \Cref{prop: SPR embedding of Cwn into Cwncupw2} is that this construction cannot be further generalized to $\alpha=\omega$, as $C[1,\omega^\omega]$ is not linearly isomorphic to $c_0$.

\begin{question}\label{quest: omega to omega}
    Is it possible to isometrically embed $C[1,\omega^\omega]$ into itself in an SPR way? 
\end{question}

We conclude this section by discussing the possibility of extending Theorem \ref{thm: SPR subspace iif K' infinite} to AM-spaces. Recall that a Banach lattice $X$ is said to be an \emph{AM-space} if $\|x\lor y\|=\max\{\|x\|,\|y\|\}$ for every pair of positive vectors $x,y\in X$. A well-known result by Kakutani states that every AM-space is lattice isometric to a (closed) sublattice of a $C(K)$-space (see \cite[Theorem 4.29]{AliprantisBurkinshaw} or \cite[Theorem 2.1.3]{MeyerNieberg}). In \cite{Eugene}, E. Bilokopytov formulates the following question:

\begin{question}{\cite[Question 5.4]{Eugene}}\label{question:eugene-Am-spaces}
     Let $X$ be an AM-space. If $X^a$, the \emph{order continuous part of $X$}, has infinite codimension in $X$, then does $X$ contain an SPR subspace of infinite dimension?
\end{question}

Note that given any compact Hausdorff space $K$, the order continuous part $C(K)^a$ has infinite codimension in $C(K)$ if and only if $K'$ is infinite (see, for instance, \cite[Remark 5.12]{BGHMT}). That is, condition $(i)$ of Theorem \ref{thm: SPR subspace iif K' infinite} can be replaced by \emph{$C(K)^a$ has infinite codimension in $C(K)$}. This means that Question \ref{question:eugene-Am-spaces} has an affirmative answer for $C(K)$-spaces. We will now give a \textit{partial answer} to this question by showing that the implication $(i)\Rightarrow(ii)$ of Theorem \ref{thm: SPR subspace iif K' infinite} does not extend to AM-spaces: we will construct an AM-space $X$ such that $X^a$ has infinite codimension in $X$ but it cannot even contain subspaces isomorphic to $c_0$ doing SPR.

Let us consider the following AM-space:
\begin{equation}\label{eq: counterexample Eugene's Q?}
X:=\left\{f\in C_0[1,\omega^2)\::\: \frac{1}{n}f(\omega\cdot(n-1)+1)=f(\omega\cdot n) \quad \text{ for all } n\geq 1\right\}.    
\end{equation}
By \cite[Proposition 2.6]{BGHMT}, $X^a$ has infinite codimension in $X$. Note that the norm-one lattice homomorphisms are: $\delta_{\omega\cdot n+m}$ for $n\geq 0$ and $m\geq 2$ are the coordinate functionals, and $n\delta_{\omega\cdot n}\equiv \delta_{\omega\cdot(n-1)+1}$, for $n\geq 1$, are the remaining norm-one lattice homomorphisms on $X$.

\begin{prop}\label{prop:c0-no-SPR-Am-espacio}
$c_0$ does not isomorphically embed into the space $X$ defined in (\ref{eq: counterexample Eugene's Q?}) in an SPR way.    
\end{prop}

\begin{proof}
Suppose that there exists an isomorphic embedding $T:c_0\to X$ with the property that $T(c_0)$ does $C$-SPR in $X$. By \Cref{thm: real stable phase retrieval}, $T(c_0)$ does not contain $C$-almost disjoint pairs, so for every natural numbers $n,m\geq 1$, $n\neq m$ we have
$$
\bigl\||Te_n|\land |Te_m|\bigr\|\geq \frac{1}{\|T^{-1}\|C},
$$
and since $X$ is an AM-space, $\text{Hom}(X,\mathbb{R})\cap S_{X^*}$, the set of normalized lattice homomorphisms from $X$ to $\R$, is $1$-norming for $X$ (see, for instance, \cite[Proposition 5.4]{BGHMT}), so for every $n,m\in\mathbb{N}$ there exists $x_{n,m}^*\in \text{Hom}(X,\mathbb{R})\cap S_{X^*}$ such that 
$$
x_{n,m}^*(|Te_n|)\land x_{n,m}^*(|Te_m|)\geq \frac{1}{\|T^{-1}\|C}.
$$

Let us see what can be said about these lattice homomorphisms $(x_{n,m}^*)_{n,m=1}^\infty$. First, note that for every $n\geq 1$ fixed, there are infinitely many distinct lattice homomorphisms in the sequence $(x_{n,m}^*)_{m>n}$. In fact, by mimicking the argument of the proof of \Cref{prop: Cw2 cannot be SPR embedded into itself}, one can check that no lattice homomorphism appears infinitely many times in this sequence (each lattice homomorphism may appear $\|T^{-1}\|\|T\|C$ different times at most).

   Now, observe that for every $n\geq 1$ fixed, there exists a natural $k_n\geq 1$ such that the set $\{x^*_{n,m}\::\: m>n\}\cap \{\delta_{t}\::\: t\in (\omega\cdot(k_n-1),\omega\cdot k_n)\}$ is infinite. Otherwise, we could find a sequence $(t_k)_{k=1}^\infty\subseteq [1,\omega^2)$ with $t_k\to \omega^2$ such that $\delta_{t_k}\in \{x_{n,m}\::\: m>n\}$ for all $k$. Since $\delta_{t_k}\overset{\omega^*}{\to} 0$ (given that the functions in $X$ vanish at $\omega^2$), we obtain that $0\in\overline{\{x_{n,m}^*\::\:m>n\}}^{w^*}$. But this is impossible, as $x_{n,m}^*(|Te_n|)\geq \frac{1}{\|T^{-1}\|\,C}$ for every $m>n$.

Therefore, for every natural $n\geq 1$, $x_n^*:=\delta_{\omega\cdot k_n}$ is an accumulation point of $(x_{n,m}^*)_{m>n}$ and $x_{n}^*(|Te_n|)\geq \frac{1}{\|T^{-1}\|\,C}$. Note that as a consequence of the definition of $X$ we have that $\|x_n^*\|=\frac{1}{k_n}$ for every $n$, so 
$$
\frac{1}{\|T^{-1}\| C}\leq x_n^*(|Te_n|)\leq \|T\|\|x_n^*\|=\frac{1}{k_n}\|T\|,  
$$
that is, $k_n\leq \|T^{-1}\|\|T\|C$ for every $n\geq 1$. On the other hand, replicating the argument of \Cref{prop: Cw2 cannot be SPR embedded into itself}, it can be proven that there must be infinitely many different $x_n^*$'s, so the sequence $(k_n)_n$ cannot be bounded, which is a contradiction.
\end{proof}

\section*{Acknowledgements}
The research of E.~Garc\'ia-S\'anchez and D.~de~Hevia is partially supported by the grants PID2020-116398GB-I00, PID2024-162214NB-I00, CEX2019-000904-S and CEX 2023-001347-S funded by the MICIU/AEI/10.13039/501100011033. 
E.~Garc\'ia-S\'anchez is partially supported by the grant CEX2019-000904-S-21-3 funded by MICIU/AEI/10. 13039/501100011033 and by ``ESF+''. D. de Hevia benefited from an FPU Grant FPU20/03334 from the Ministerio de Ciencia, Innovación y Universidades. \\

The authors are especially grateful to Mitchell A.~Taylor for his tutoring and mentoring role throughout the project. The authors also want to thank Antonio Avilés, Eugene Bilokopytov, Jesús Illescas, Timur Oikhberg, Alberto Salguero and Pedro Tradacete for their valuable discussions and suggestions.\\

After this work was completed, Jakub Rondo\v{s} and Damian Sobota kindly pointed out their paper \cite{RS}, where they studied a similar but weaker property than property \eqref{property *}. Their paper was completed independently and in parallel with the development of our arguments (see version 1 of our paper on arXiv). In this new edition, we have incorporated ideas from both papers to make a clearer presentation of our results. We thank Jakub Rondo\v{s} and Damian Sobota for bringing their paper to our attention.

\end{document}